\documentclass[dvipdfmx]{article}
\usepackage{amsmath,graphicx,amssymb,amsthm,bm,comment,color,url}
\makeatletter

\makeatletter

\@addtoreset{equation}{section}
\makeatother

\usepackage{geometry}
\numberwithin{equation}{section}
\newtheorem{theorem}{Theorem}[section]
\newtheorem{corollary}[theorem]{Corollary}
\newtheorem{lemma}[theorem]{Lemma}

\newtheorem{remark}{Remark}

\def\E{\mathsf{E}}
\def\e{\mathrm{e}}

\title{Evaluating moments of length of Pitman partition}
\date{}
\author{Koji Tsukuda\footnote{Faculty of Mathematics, Kyushu University, 744 Motooka, Nishi-ku, Fukuoka-shi, Fukuoka 819-0395, Japan.}}

\begin{document}
\maketitle

\allowdisplaybreaks[3]

\begin{abstract}\noindent
The Pitman sampling formula has been intensively studied as a distribution of random partitions.
One of the objects of interest is the length $K (= K_{n,\theta,\alpha})$ of a random partition that follows the Pitman sampling formula, where $n\in\mathbb{N}$, $\alpha\in(0,\infty)$ and $\theta > -\alpha$ are parameters.
This paper presents asymptotic evaluations of its $r$-th moment $\mathsf{E}[K^r]$ ($r=1,2,\ldots$) under two asymptotic regimes.
In particular, the goals of this study are to provide a finer approximate evaluation of $\mathsf{E}[K^r]$ as $n\to\infty$ than has previously been developed and to provide an approximate evaluation of $\mathsf{E}[K^r]$ as the parameters $n$ and $\theta$ simultaneously tend to infinity with $\theta/n \to 0$.
The results presented in this paper will provide a more accurate understanding of the asymptotic behavior of $K$.
\end{abstract}

\vspace{5truemm}
\par
\textbf{MSC2010 subject classifications}: 60C05, 60F05, 62E20.

\section{Introduction}\label{sec1}

Let  $n$ be a positive integer, and let $c_1,\ldots,c_n$ be the component counts of an integer partition of $n$, which means that the partition has $c_i$ parts of size $i$ $(i=1,\ldots,n)$.
Note that $c_1,\ldots,c_n$ satisfy $n=\sum_{i=1}^n i c_i$.
Define a set $\mathcal{P}_n=\{ (c_1,\ldots,c_n) \in (\mathbb{N}\cup\{0\})^{n}; \sum_{i=1}^n i c_i = n \}$.
Let $\alpha\in[0,1)$, $\theta\in(-\alpha,\infty)$.
A $\mathcal{P}_n$-valued random variable $\textbf{C} = (C^{n,\theta,\alpha}_1,\ldots,C^{n,\theta,\alpha}_n)$ is component counts of a \textit{Pitman partition} when its distribution is given by
\begin{equation}\label{PSF}
 {\sf P} (\textbf{C} = (c_1,\ldots,c_n)) = n! \frac{(\theta)_{k; \alpha}}{(\theta)_n} \prod_{i=1}^n \left( \frac{(1-\alpha)_{i-1} }{i!} \right)^{c_i} \frac{1}{c_i!} \quad ( (c_1,\ldots,c_n) \in \mathcal{P}_n) ,
\end{equation}
where $k=\sum_{i=1}^n c_i$,
\[ (x)_{i; y}=
\begin{cases}
1 & (i=0) \\
x(x+ y) (x+2 y) \cdots (x+(i-1) y) & (i=1,2,\ldots)
\end{cases}
 \quad (x > -y; y \geq 0), \]
and $(x)_{i}=(x)_{i;1}$ $(i=0,1,2,\ldots; x> -1)$.
The distribution \eqref{PSF}, called the \textit{Pitman sampling formula}, was introduced by Jim Pitman; see, for example, \cite{RefP95,RefP06}.
A special case $\alpha=0$ of \eqref{PSF} is known as the Ewens sampling formula, which was introduced by Ewens~\cite{RefE} in the context of population genetics.
Henceforth, unless otherwise mentioned, we consider the case $\alpha\neq0$.
The distribution of the length $K (=K_{n,\theta,\alpha}) = \sum_{i=1}^n C^{n,\theta,\alpha}_i$ of a Pitman partition is given by
\begin{equation}\label{PL}
 {\sf P}(K= k ) =  \frac{c(n,k,\alpha)}{\alpha^k}  \frac{(\theta)_{k; \alpha}}{(\theta)_n} \quad (k=1,\ldots,n), 
\end{equation}
where $c(n,k,\alpha)=(-1)^{n-k} {\rm C}(n,k,\alpha)$ and ${\rm C}(n,k,\alpha)$ is the generalized Stirling number or the C-number of Charalambides and Singh~\cite{RefCS}; see, for example, Yamato, Sibuya and Nomachi~\cite{RefYSN}.
This paper discusses an asymptotic property of $K$.

Random partition models have been received considerable attention, not only because they are interesting as mathematical models, but also because they relate to broad scientific fields; see, for example, Crane~\cite{RefCr} and Johnson, Kotz and Balakrishnan~\cite{RefJKB} (Chapter 41, its write-up was provided by S.~Tavar\'e and W.J.~Ewens).
As a typical distribution of random partiton models, the distribution \eqref{PSF} has been intensively studied.
In particular, the properties of \eqref{PL} have been investigated by Dolera and Favaro~\cite{RefDF}, Favaro, Feng and Gao~\cite{RefFFG}, Feng and Hoppe~\cite{RefFH}, Pitman~\cite{RefP97,RefP99}, Yamato and Sibuya~\cite{RefYS}, and Yamato, Sibuya and Nomachi~\cite{RefYSN}. 
A prominent limit theorem associated with \eqref{PL} is
\begin{equation}\label{cid}
 \frac{K}{n^\alpha} \Rightarrow S_{\alpha,\theta} 
\end{equation}
as $n\to \infty$, where $\Rightarrow$ denotes convergence in distribution and the limit random variable $S_{\alpha,\theta}$, called the \textit{Pitman $\alpha$-diversity}, is a continuous random variable whose distribution has the density
\[ \frac{\Gamma(\theta+1)}{\Gamma(\frac{\theta}{\alpha}+1)} x^{\frac{\theta}{\alpha}} g_\alpha(x) \quad ( x >0 ), \]
where $g_\alpha(\cdot)$ is the function satisfying
\[ \int_0^\infty x^p g_\alpha (x) dx = \frac{\Gamma(p+1)}{\Gamma(p\alpha +1)} \quad (p>-1); \]
see, for example, Pitman~\cite{RefP97,RefP99} or Yamato and Sibuya~\cite{RefYS}.
Note that $g_\alpha(x)$ can be written as
\[
g_{\alpha}(x) = \frac{1}{\pi \alpha} \sum_{i=1}^{\infty} \frac{(-1)^{i+1}}{i!} \Gamma(i \alpha +1) x^{i-1} \sin(\pi i \alpha ) \quad (x>0).
\]
A summary of the proof of \eqref{cid} given by Yamato and Sibuya~\cite{RefYS} is as follows: 
\begin{enumerate}
\item For $r=1,2,\ldots$, it holds that
\begin{equation}\label{YS1}
{\sf E} \left[ K^r \right] = \sum_{i=0}^r (-1)^{r-i} \left( 1 +\frac{\theta}{\alpha} \right)_i R \left(r, i, \frac{\theta}{\alpha}\right) \frac{\Gamma(\theta + i \alpha +n) }{\Gamma(\theta+n) } \frac{\Gamma(\theta+1)}{\Gamma(\theta+i \alpha +1)},
\end{equation}
where $R(\cdot,\cdot,\cdot)$ is the weighted Stirling number of the second kind introduced by Carlitz~\cite{RefC}.
\item Using expression \eqref{YS1}, the Stirling formula yields
\begin{equation}\label{YS2}
{\sf E}\left[ \left( \frac{K}{n^\alpha} \right)^r \right] = \left(1 + \frac{\theta}{\alpha} \right)_{r} \frac{\Gamma(\theta+1)}{\Gamma(\theta+r\alpha+1)}  + o(1)
\quad (r=1,2,\ldots)
\end{equation}
as $n\to\infty$.
\item Hence, \eqref{cid} follows from the method of moments.
\end{enumerate}

In this paper, the moments $\E[K^r]$ ($r=1,2,\ldots$) are investigated in more detail than \eqref{YS2}. 
In particular, there are two goals in this paper: the first is to provide an approximate evaluation of $\E[K^r]$ as $n\to\infty$ that is finer than that of \eqref{YS2}; the second is to provide an approximate evaluation of $\E[K^r]$ as the parameters $n$ and $\theta$ simultaneously tend to infinity with $\theta/n \to 0$.
The phenomenon described by \eqref{cid} is attractive, and our results will provide a more accurate understanding of \eqref{cid}.

\begin{remark}
There are some known results stronger than \eqref{cid}.
In particular,
\begin{itemize}
\item $K/n^\alpha$ converges to $S_{\alpha,\theta}$ almost surely and in $p$-th mean for any $p>0$.
\item A Berry--Esseen-type theorem holds: When $\alpha\in(0,1)$ and $\theta >5$, $\sup_{x\geq 0}|\mathsf{P}(K/n^\alpha \leq x) - \mathsf{P}(S_{\alpha,\theta}\leq x)| \leq {C(\alpha,\theta)}/{n^\alpha}$ holds for $n\in\mathbb{N}$, where $C(\alpha,\theta)$ is a constant depending only on $\alpha$ and $\theta$.
\end{itemize}
For details, see Dolera and Favaro~\cite{RefDF} and Pitman~\cite{RefP06}. 
\end{remark}

\section{Asymptotic regime}
In this paper, we consider two asymptotic regimes.
The first is $n\to\infty$ with fixed $\theta$, and the second is 
\begin{equation}\label{AR}
n\to\infty, \ \theta \to \infty,\ \frac{\theta}{n} \to 0.
\end{equation}
The former ($n\to\infty$ with fixed $\theta$) has been frequently considered.
When $\alpha=0$, \eqref{AR} has also been considered in some studies; see Remark~\ref{rems2} below.
However, when $\alpha\neq0$, \eqref{AR} has not been considered, although it also seems natural.
Indeed, in the application of \eqref{PSF} to microdata risk assessment by Hoshino~\cite{RefHo}, the estimates of $\theta$ take large values (e.g., $523, \ldots, 21298 $, where $n = 27320$); see Tables 3--6 of \cite{RefHo}.
Throughout the paper, we assume that $\theta>0$ when considering the regime \eqref{AR}.

\begin{remark}\label{rems2}
When $\alpha=0$, the asymptotic regime in which $n$ and $\theta$ simultaneously tend to infinity has been considered by Feng~\cite{RefFa} and Tsukuda~\cite{RefT17,RefT19,RefT20}.
In particular, \eqref{AR} is Case D in Section 4 of \cite{RefFa}.
\end{remark}

\begin{remark}
For the Pitman sampling formula, the asymptotic regime $\theta\to\infty$ with fixed $n$ was considered by Kerov~\cite{RefK}.
Moreover, Feng~\cite{RefFb} considered the asymptotic regime $\theta\to\infty$ in studying the Pitman--Yor process, which is closely related to the Pitman sampling formula.
For details of the Pitman--Yor process, see, for example, Pitman and Yor~\cite{RefPY}.
\end{remark}

\section{Results}

\subsection{Asymptotic evaluation as $n\to\infty$ with fixed $\theta$}
First, we provide an evaluation that is finer than that of \eqref{YS2} under the asymptotic regime $n\to\infty$ with fixed $\theta$.

\begin{theorem}
\label{thm1}
Suppose that $\alpha>0$.
For $r=1,2,\ldots$, it holds that
\begin{eqnarray*}
{\sf E}\left[ \left( \frac{K}{n^\alpha} \right)^r \right]
&=& \left(1+ \frac{\theta}{\alpha} \right)_{r} \frac{\Gamma(\theta+1)}{\Gamma(\theta+r\alpha+1) } \left[
 1
- \left\{ \frac{r(r-1)\alpha}{2} + r\theta \right\} \frac{\Gamma(\theta+r\alpha)}{\Gamma(\theta+(r-1)\alpha+1)} \frac{1}{n^\alpha} 
\right]
\\&&\quad
+ O\left( \frac{1}{n^{2\alpha}} + \frac{1}{n} \right)  
\end{eqnarray*}
as $n\to\infty$.
\end{theorem}

\begin{proof}
It follows from \eqref{YS1} that
\begin{eqnarray}
{\sf E}\left[ K^r \right] &=& 
\sum_{i=0}^{r-2} (-1)^{r-i} \left( 1 +\frac{\theta}{\alpha} \right)_i R \left(r, i, \frac{\theta}{\alpha}\right) \frac{\Gamma(\theta + i \alpha +n) }{\Gamma(\theta+n) } \frac{\Gamma(\theta+1)}{\Gamma(\theta+i \alpha +1)} \nonumber \\
&& - \left( 1 +\frac{\theta}{\alpha} \right)_{r-1} R \left(r, r-1, \frac{\theta}{\alpha}\right) \frac{\Gamma(\theta + (r-1) \alpha +n) }{\Gamma(\theta+n) } \frac{\Gamma(\theta+1)}{\Gamma(\theta+ (r-1) \alpha +1)} \nonumber \\
&& + \left( 1 +\frac{\theta}{\alpha} \right)_r R \left(r, r, \frac{\theta}{\alpha}\right) \frac{\Gamma(\theta + r \alpha +n) }{\Gamma(\theta+n) } \frac{\Gamma(\theta+1)}{\Gamma(\theta+ r \alpha +1)},
\label{YS1a}
\end{eqnarray}
where $\sum_{i=0}^{-1} a_i = 0$ for any sequence $\{a_1,a_2,\ldots \}$.
As 
\begin{equation}\label{stir}
 \Gamma(x) = \sqrt{2\pi} \e^{-x} x^{x-1/2} \left(1+ \frac{1}{12x} + O\left( \frac{1}{x^2} \right) \right) \quad (x\to\infty),
\end{equation}
it holds that
\[ \frac{\Gamma(\theta + i\alpha +n)}{\Gamma(\theta + n)} 
=  n^{i \alpha} \left[ 1 + \frac{i \alpha \{n(2\theta + i \alpha -1) + 2\theta^2 \}}{2n (n+\theta)} + O\left( \frac{1}{n^2} \right) \right]
= n^{i \alpha} \left( 1 + O\left( \frac{1}{n} \right) \right) \]
for $i=0,1,\ldots$.
The first term on the right-hand side of \eqref{YS1a} is $O(n^{(r-2)\alpha})$.
The second term on the right-hand side of \eqref{YS1a} is
\begin{eqnarray*}
\lefteqn{ - \left( 1 +\frac{\theta}{\alpha} \right)_{r-1} \left\{ \frac{r(r-1)}{2} + \frac{r \theta}{\alpha} \right\} \frac{ \Gamma(\theta+1)}{ \Gamma(\theta+ (r-1) \alpha +1)} n^{(r-1)\alpha} \left( 1 + O\left( \frac{1}{n} \right)  \right)  } \\
&=&  - \left( 1 +\frac{\theta}{\alpha} \right)_{r} \frac{\alpha}{\theta+\alpha r} \left\{ \frac{r(r-1)}{2} + \frac{r \theta}{\alpha} \right\} \frac{ \Gamma(\theta+1)}{ \Gamma(\theta+ (r-1) \alpha +1)} n^{(r-1)\alpha} +  O\left( n^{(r-1)\alpha -1} \right)  \\
&=&  - \left( 1 +\frac{\theta}{\alpha} \right)_{r} \frac{\Gamma(\theta+1)  }{ \Gamma(\theta+r \alpha +1) } \left\{ \frac{r(r-1) \alpha}{2} + r \theta \right\} \frac{ \Gamma(\theta+r \alpha)}{ \Gamma(\theta+ (r-1) \alpha +1)} n^{(r-1)\alpha} 
\\&& \quad
+  O\left( n^{(r-1)\alpha -1} \right),
\end{eqnarray*}
because (3.2) of Carlitz~\cite{RefC} implies that
\[ R \left( r, r-1 , \frac{\theta}{\alpha} \right) 
= \sum_{i=0}^{1} \binom{r}{i} \left( \frac{\theta}{\alpha} \right)^i S_2(r-i, r-1)
= S_2(r,r-1)+ \frac{r \theta}{\alpha} = \frac{r(r-1)}{2} + \frac{r\theta}{\alpha} , \]
where $S_2(\cdot,\cdot)$ is the Stirling number of the second kind.
The third term on the right-hand side of \eqref{YS1a} is
\[
\left( 1 +\frac{\theta}{\alpha} \right)_r  \frac{\Gamma(\theta+1)}{ \Gamma(\theta+ r \alpha +1)} n^{r \alpha} + O \left( n^{r \alpha -1} \right).
\]
As $\alpha\in(0,1)$, we have that
\[ O(n^{(r-2)\alpha}) +  O\left( n^{(r-1)\alpha -1} \right) +  O \left( n^{r \alpha -1} \right)  = O\left( n^{(r-2)\alpha} + n^{r\alpha-1} \right). \]
Therefore,
\begin{eqnarray*}
&& {\sf E}\left[ \left( \frac{K}{n^\alpha} \right)^r \right]  \\
&&= 
\left( 1 +\frac{\theta}{\alpha} \right)_r  \frac{\Gamma(\theta+1)}{ \Gamma(\theta+ r \alpha +1)}
\\ && \quad
- \left( 1 +\frac{\theta}{\alpha} \right)_{r} \frac{\Gamma(\theta+1)  }{ \Gamma(\theta+r \alpha +1) } \left\{ \frac{r(r-1) \alpha}{2} + r \theta \right\} \frac{ \Gamma(\theta+r \alpha)}{ \Gamma(\theta+ (r-1) \alpha +1)} \frac{1}{n^\alpha} 
\\ && \quad
+ O\left( \frac{1}{n^{2\alpha}} + \frac{1}{n} \right)   \\
&&= 
\left(1+ \frac{\theta}{\alpha} \right)_{r} \frac{\Gamma(\theta+1)}{\Gamma(\theta+r\alpha+1) } \left[
 1
- \left\{ \frac{r(r-1)\alpha}{2} + r\theta \right\} \frac{\Gamma(\theta+r\alpha)}{\Gamma(\theta+(r-1)\alpha+1)} \frac{1}{n^\alpha} 
\right]
\\ && \quad
+ O\left( \frac{1}{n^{2\alpha}} + \frac{1}{n} \right)  .
\end{eqnarray*}
This completes the proof.
\end{proof}

\begin{remark}\label{remC}
When $r > 3$, as $\theta > -\alpha$, it holds that
\[ \left\{ \frac{r(r-1) \alpha }{2} + r\theta \right\} \frac{\Gamma(\theta+r\alpha)}{\Gamma(\theta+(r-1)\alpha +1) }\frac{1}{n^\alpha} > 0. \]
This means that almost all moments of $K/n^\alpha$ are smaller than those of $S_{\alpha,\theta}$ for large $n$.
In particular, if $\theta>0$, then all moments of $K/n^\alpha$ are smaller than those of $S_{\alpha,\theta}$ for large $n$.
\end{remark}

In some cases, correcting some moments improves the quality of an approximation.
Thus, a primitive application of Theorem~\ref{thm1} is replacing $K/n^\alpha$ in \eqref{cid} by 
\[
\frac{K}{n^\alpha - \frac{\theta \Gamma(\theta+ \alpha) }{ \Gamma(\theta+1) }  }
\]
whose expectation is
\[ 
{\sf E}\left[ \frac{K}{n^\alpha - \frac{\theta \Gamma(\theta+ \alpha) }{ \Gamma(\theta+1) }  } \right] 
=  \frac{\Gamma(\theta+1)}{\alpha \Gamma(\theta+\alpha)}  + O \left( \frac{1}{n^{2\alpha}} + \frac{1}{n} \right).
 \]
When $\theta>0$, this correction enlarges $K/n^\alpha$, and is consistent with Remark~\ref{remC}.

\subsection{Asymptotic evaluation under \eqref{AR}}
Next, under the asymptotic regime of \eqref{AR}, we provide a new evaluation.

\begin{theorem}
\label{mthA}
Suppose that $\alpha>0$.
For $r=1,2,\ldots$, it holds that
\[
 {\sf E}\left[ \left( \frac{\alpha K}{\theta \left\{ (\frac{n+\theta}{\theta})^\alpha - 1 \right\} } \right)^r \right] 
= 1 + O \left(  \frac{\theta^{2\alpha}}{n^{2\alpha}} +  \frac{\theta}{n} + \frac{1}{\theta} \right)
\]
under the asymptotic regime of \eqref{AR}; 
in particular, for $r=1,2,\ldots$, under the asymptotic regime of \eqref{AR}, if 
\[ \frac{\theta^{2\alpha+1}}{n^{2\alpha}} \to 0 \quad {\rm and} \quad
 \frac{\theta^2}{n} \to 0, \]
then
\[
 {\sf E}\left[ \left( \frac{\alpha K}{\theta \left\{ (\frac{n+\theta}{\theta})^\alpha - 1 \right\} } \right)^r \right] 
= 1 + r^2 \frac{\alpha(1-\alpha)}{2 \theta} 
+O \left(  \frac{\theta^{2\alpha}}{n^{2\alpha}} +  \frac{\theta}{n} + \frac{1}{\theta}\left(  \frac{\theta^\alpha}{n^\alpha } + \frac{1}{\theta} \right) \right).
\]
\end{theorem}

Before presenting the proof of Theorem~\ref{mthA}, let us prepare the following lemma.

\begin{lemma}\label{kle}
Suppose that $\alpha>0$.
For $r=1,2,\ldots$, it holds that
\[ 
{\sf E}\left[ \left\{ \frac{\alpha K}{\theta (\frac{n}{\theta})^\alpha } \right\}^r \right] 
= 1 - r\left( \frac{\theta}{n} \right)^\alpha + r^2 \frac{\alpha(1-\alpha)}{2 \theta} 
+O \left(  \frac{\theta^{\alpha}}{n^{\alpha}}\left( \frac{\theta^{\alpha}}{n^{\alpha}} + \frac{\theta^{1-\alpha}}{n^{1-\alpha}} \right) + \frac{1}{\theta}\left(  \frac{\theta^\alpha}{n^\alpha } + \frac{1}{\theta} \right) \right)\]
under the asymptotic regime of \eqref{AR}.
\end{lemma}

\begin{proof}
It follows from \eqref{YS1} that
\begin{eqnarray*}
{\sf E}\left[ K^r  \right] 
&=&\sum_{i=0}^{r-2} (-1)^{r-i} \left( 1 +\frac{\theta}{\alpha} \right)_i R \left(r, i, \frac{\theta}{\alpha}\right) \frac{\Gamma(\theta + i \alpha +n) \Gamma(\theta+1)}{\Gamma(\theta+n) \Gamma(\theta+i \alpha +1)} \\&&
 - \left( 1 +\frac{\theta}{\alpha} \right)_{r-1} \left\{ \frac{r(r-1)}{2} 
+ \frac{r\theta}{\alpha}  \right\} \frac{\Gamma(\theta + (r-1) \alpha +n) \Gamma(\theta+1)}{\Gamma(\theta+n) \Gamma(\theta+ (r-1) \alpha +1)} 
\\&& 
+ \left( 1 +\frac{\theta}{\alpha} \right)_r \frac{\Gamma(\theta + r \alpha +n) \Gamma(\theta+1)}{\Gamma(\theta+n) \Gamma(\theta+ r \alpha +1)} \\
&=& \sum_{i=0}^{r-2} (-1)^{r-i}  R \left(r, i, \frac{\theta}{\alpha}\right) \frac{\Gamma\left( \frac{\theta}{\alpha} +1 + i \right)}{\Gamma\left( \frac{\theta}{\alpha} +1 \right)} \frac{\Gamma(\theta + n  + i \alpha ) }{\Gamma(\theta+n) } \frac{\Gamma(\theta+1)}{\Gamma(\theta + 1 +i \alpha)} 
\\&& 
- \left\{ \frac{r(r-1)}{2} + \frac{r\theta}{\alpha}  \right\} \frac{\Gamma\left( \frac{\theta}{\alpha} + r \right)}{\Gamma\left( \frac{\theta}{\alpha} +1 \right)} \frac{\Gamma(\theta + n  + (r-1) \alpha ) }{\Gamma(\theta+n) } \frac{\Gamma(\theta+1)}{\Gamma(\theta + 1 + (r-1) \alpha)}  
\\&&
 + \frac{\Gamma\left( \frac{\theta}{\alpha} +1 + r \right)}{\Gamma\left( \frac{\theta}{\alpha} +1 \right)}  \frac{\Gamma(\theta + n  + r \alpha ) }{\Gamma(\theta+n) } \frac{\Gamma(\theta+1)}{\Gamma(\theta + 1 +r \alpha)} .
\end{eqnarray*}
As $R(r,i,\theta/\alpha) = O(\theta^{r-i})$ for $i=0,1,\ldots,r-2$, which follows from (3.2) of \cite{RefC}, according to Lemma~\ref{keyp}, the first term is  
\[ 
O\left( \theta^2 \left(\theta \frac{n^\alpha}{\theta^\alpha} \right)^{r-2}  \right)
= O\left( \left(\theta \frac{n^\alpha}{\theta^\alpha} \right)^r \frac{\theta^{2\alpha}}{n^{2\alpha}} \right) \]
that stems from the term of $i=r-2$.
Moreover, using Lemma~\ref{keyp} again, the second term is
\begin{eqnarray*}
&&
-\left( O(1) + \frac{r\theta}{\alpha} \right) \left( \frac{\theta}{\alpha} \frac{n^\alpha}{\theta^\alpha} \right)^{r-1} \left(1+ O\left( \frac{\theta}{n} + \frac{1}{\theta} \right) \right) 
\\&&
=  -\left( \frac{\theta}{\alpha} \frac{n^\alpha}{\theta^\alpha} \right)^{r}  \left\{ r \left( \frac{\theta}{n} \right)^\alpha + O\left( \frac{\theta^{\alpha+1}}{n^{\alpha+1}} + \frac{1}{n^\alpha \theta^{1-\alpha}} \right) \right\}
\end{eqnarray*}
and the third term is
\begin{eqnarray*}
&&
\left( \frac{\theta}{\alpha} \frac{n^\alpha}{\theta^\alpha} \right)^r 
\left\{ 1+ \frac{r \alpha \theta}{n} + \frac{r^2 \alpha(1-\alpha)}{2\theta} + O \left( \frac{1}{\theta^2} + \frac{\theta^2}{n^2} \right)  \right\}
\\&&
=\left( \frac{\theta}{\alpha} \frac{n^\alpha}{\theta^\alpha} \right)^r 
\left\{ 1+ \frac{r^2 \alpha(1-\alpha)}{2\theta} + O \left( \frac{1}{\theta^2} + \frac{\theta}{n} \right)  \right\}  .
\end{eqnarray*}
These formulae yield
\[
{\sf E}\left[ K^r  \right] 
= \left( \frac{\theta}{\alpha} \frac{n^\alpha}{\theta^\alpha} \right)^r \left\{ 1 - r \left( \frac{\theta}{n} \right)^\alpha + \frac{r^2 \alpha(1-\alpha)}{2\theta}  +O \left( \frac{\theta^{2\alpha}}{n^{2\alpha}} + \frac{1}{n^\alpha \theta^{1-\alpha}} + \frac{1}{\theta^2} + \frac{\theta}{n} \right)  \right\},
\]
where
\[ 
O \left(\frac{\theta^{2\alpha}}{n^{2\alpha}} + \frac{\theta^{\alpha+1}}{n^{\alpha+1}}+ \frac{1}{n^\alpha \theta^{1-\alpha}} + \frac{1}{\theta^2} + \frac{\theta}{n}  \right)  
= O \left( \frac{\theta^{2\alpha}}{n^{2\alpha}}+ \frac{1}{n^\alpha \theta^{1-\alpha}} + \frac{1}{\theta^2} + \frac{\theta}{n} \right)
\]
is used.
Thus, it holds that
\[
{\sf E}\left[ \left\{ \frac{\alpha K}{\theta (\frac{n}{\theta})^\alpha} \right\}^r  \right] 
= 1 - r \left( \frac{\theta}{n} \right)^\alpha + \frac{r^2 \alpha(1-\alpha)}{2\theta} 
 +  O \left( \frac{\theta^{2\alpha}}{n^{2\alpha}} + \frac{1}{n^\alpha \theta^{1-\alpha}}+ \frac{1}{\theta^2}+ \frac{\theta}{n} \right).
\]
This completes the proof.
\end{proof}

Using Lemma~\ref{kle}, we prove Theorem~\ref{mthA}.

\begin{proof}[Proof of Theorem~\ref{mthA}]
It follows from
\[ 
\frac{( \frac{n+\theta}{\theta} )^\alpha -1}{(\frac{n}{\theta})^\alpha}
= \left( 1 + \frac{\theta}{n} \right)^\alpha - \left( \frac{\theta}{n} \right)^\alpha
= 1 - \left( \frac{\theta}{n} \right)^\alpha  + O\left( \frac{\theta}{n} \right) 
\]
that
\begin{equation}\label{t322a}
\left\{\frac{(\frac{n}{\theta})^\alpha}{( \frac{n+\theta}{\theta} )^\alpha -1}\right\}^{r} 
= 1  + r \left(\frac{\theta}{n}\right)^\alpha +  O\left( \frac{\theta^{2\alpha}}{n^{2\alpha}} +  \frac{\theta}{n} \right) .
\end{equation}
Lemma~\ref{kle} and \eqref{t322a} yield
\begin{eqnarray*}
\lefteqn{
{\sf E}\left[ \left( \frac{\alpha K}{\theta \left\{ (\frac{n+\theta}{\theta})^\alpha - 1 \right\} } \right)^r \right] 
= \left\{\frac{(\frac{n}{\theta})^\alpha}{( \frac{n+\theta}{\theta} )^\alpha -1}\right\}^{r} {\sf E}\left[ \left\{ \frac{\alpha K}{\theta (\frac{n}{\theta})^\alpha } \right\}^r \right] 
} \\
&=& \left\{ 1  + r \left(\frac{\theta}{n}\right)^\alpha +  O\left( \frac{\theta^{2\alpha}}{n^{2\alpha}} +  \frac{\theta}{n} \right) \right\} 
\\ && \times
\left\{ 1 - r\left( \frac{\theta}{n} \right)^\alpha + r^2 \frac{\alpha(1-\alpha)}{2 \theta} 
+O \left(  \frac{\theta^{\alpha}}{n^{\alpha}}\left( \frac{\theta^{\alpha}}{n^{\alpha}} + \frac{\theta^{1-\alpha}}{n^{1-\alpha}} \right) + \frac{1}{\theta}\left(  \frac{\theta^\alpha}{n^\alpha } + \frac{1}{\theta} \right) \right) \right\} \\
&=&   1 + r^2 \frac{\alpha(1-\alpha)}{2 \theta} 
+O \left(  \frac{\theta^{2\alpha}}{n^{2\alpha}} +  \frac{\theta}{n} + \frac{1}{\theta}\left(  \frac{\theta^\alpha}{n^\alpha } + \frac{1}{\theta} \right) \right).
\end{eqnarray*}
It implies the desired conclusion.
This completes the proof.
\end{proof}

\begin{corollary}
\label{cor34}
Suppose that $\alpha>0$.
It holds that
\begin{equation}\label{coreq}
\frac{\alpha K}{\theta \left\{ (\frac{n+\theta}{\theta})^\alpha - 1 \right\} } \to^p 1
\end{equation}
under the asymptotic regime of \eqref{AR}, where $\to^p$ denotes convergence in probability.
\end{corollary}

\begin{proof}
Using Theorem~\ref{mthA}, the first and second moments of the left-hand side in \eqref{coreq} converge to 1.
This completes the proof.
\end{proof}

This corollary shows that, under the asymptotic regime of \eqref{AR}, $K$ may asymptotically behave as if $\alpha$ was 0 from the perspective of the following remark.

\begin{remark}
When $\alpha=0$, it is known that
\begin{equation}\label{correm}
\frac{K_{n,\theta,0} }{\theta \log\left(1+ \frac{n}{\theta} \right)} \to^p 1 
\end{equation}
under the asymptotic regime of \eqref{AR}; see \cite{RefFa,RefT17}.
Hence, it holds that
\[  \frac{\alpha K_{n,\theta,\alpha}}{\theta \left\{ (\frac{n+\theta}{\theta})^\alpha - 1 \right\} } 
\underset{\alpha\to +0}{\Rightarrow} \frac{K_{n,\theta,0}}{\theta \log\left(1+ \frac{n}{\theta} \right)} \underset{\eqref{AR}}{\to^p} 1. \]
Corollary~\ref{cor34} shows the same limit without the first operation $\alpha \to +0$.
In this sense, \eqref{coreq} complements the previous result of \eqref{correm}.
\end{remark}

\section{Technical lemma}

In this section, we prove the following lemma, which was used in the proof of Lemma~\ref{kle}.

\begin{lemma}\label{keyp}
Under the asymptotic regime of \eqref{AR}, it holds that
\begin{eqnarray*}
&&
\frac{\Gamma\left(\frac{\theta}{\alpha}+1+i\right)}{\Gamma\left(\frac{\theta}{\alpha}+1\right)}
\frac{\Gamma(\theta+n+i\alpha)}{\Gamma(\theta+n)}
\frac{\Gamma(\theta+1)}{\Gamma(\theta+1+i\alpha)} 
\\&&
=
\left( \frac{\theta}{\alpha} \frac{n^\alpha}{\theta^\alpha} \right)^i \left\{ 1+ \frac{i \alpha \theta}{n} + \frac{i^2 \alpha(1-\alpha)}{2\theta} + O \left( \frac{1}{\theta^2} + \frac{\theta^2}{n^2} \right)  \right\}
\end{eqnarray*}
for $i=1,2,\ldots$.
\end{lemma}

To prove this assertion, we first prove the following three lemmas.

\begin{lemma}\label{lem1}
Under the asymptotic regime \eqref{AR}, it holds that
\[
\frac{\Gamma(\theta+n+i\alpha)}{\Gamma(\theta+n)}
= n^{i \alpha} \left( 1 + \frac{i \alpha \theta}{n} + O\left(\frac{1}{n} + \frac{\theta^2}{n^2}\right) \right) 
\]
for $i=1,2,\ldots$.
\end{lemma}

\begin{proof}
It follows from \eqref{stir} that
\begin{eqnarray*}
&&  \frac{\Gamma(\theta+n+i\alpha)}{\Gamma(\theta+n)} \\
&&= \frac{\e^{-(\theta+n+i \alpha)} (\theta+n +i\alpha)^{\theta+n+i\alpha-1/2} \left\{ 1 + \frac{1}{12(\theta + i\alpha +n)} + O\left( \frac{1}{(\theta+n)^2} \right) \right\} }{ \e^{-(\theta+n)} (\theta+n )^{\theta+n-1/2} \left\{1 + \frac{1}{12(\theta  +n)} + O\left( \frac{1}{(\theta+n)^2} \right) \right\} } \\
&&= 
\e^{-i\alpha} (\theta+n)^{i \alpha} \left(1 + \frac{i\alpha}{\theta+n} \right)^{i\alpha - 1/2} \left( 1 + \frac{i \alpha}{\theta+n} \right)^{\theta+n}  \left\{ \frac{ 1 + \frac{1}{12(\theta + i\alpha +n)} + O\left( \frac{1}{(\theta+n)^2} \right) }{   1 + \frac{1}{12(\theta +n)} + O\left( \frac{1}{(\theta+n)^2} \right) } \right\} .
\end{eqnarray*}
Hence, it holds that
\begin{eqnarray*}
&&
 \frac{\Gamma(\theta+n+i\alpha)}{\Gamma(\theta+n)} \\
&&=  
\e^{-i\alpha} (\theta+n)^{i \alpha} 
\left\{ 1 + \frac{i\alpha}{\theta+n}\left(i\alpha - \frac{1}{2} \right) + O\left( \frac{1}{(\theta+n)^2} \right) \right\} 
\e^{i \alpha} 
\\ && \quad \times
\left\{ 1 - \frac{i^2 \alpha^2}{2(\theta+n)}  + O\left( \frac{1}{(\theta+n)^2} \right) \right\} 
\left\{  1 + \frac{1}{12(\theta + i\alpha +n)} + O\left( \frac{1}{(\theta+n)^2} \right)  \right\} 
\\ &&\quad \times
\left\{   1 - \frac{1}{12(\theta +n)} + O\left( \frac{1}{(\theta+n)^2} \right)  \right\} \\
&&=  
 (\theta+n)^{i \alpha} 
\left\{ 1 + \frac{i\alpha}{\theta+n}\left(i\alpha - \frac{1}{2} \right) + O\left( \frac{1}{(\theta+n)^2} \right) \right\} 
\left\{ 1 - \frac{i^2 \alpha^2 }{2(\theta+n)}  + O\left( \frac{1}{(\theta+n)^2} \right) \right\}
\\ &&\quad \times
  \left\{  1 + \frac{1}{12(\theta + i\alpha +n)} - \frac{1}{12(\theta +n)}  + O\left( \frac{1}{(\theta+n)^2} \right)  \right\} \\
&&=  
 (\theta+n)^{i \alpha} 
\left\{ 1 + \frac{i\alpha}{\theta+n}\left(i\alpha - \frac{1}{2} \right) + O\left( \frac{1}{(\theta+n)^2} \right) \right\} 
\left\{ 1 - \frac{i^2 \alpha^2 }{2(\theta+n)}  + O\left( \frac{1}{(\theta+n)^2} \right) \right\}
\\ && \quad \times
 \left\{  1 + O\left( \frac{1}{(\theta+n)^2} \right)  \right\}\\
&&= (\theta+n)^{i\alpha} \left\{ 1+ \frac{i \alpha  (i \alpha -1)}{2(\theta+n)} + O\left( \frac{1}{(\theta+n)^2} \right)  \right\}.
\end{eqnarray*}
From $\theta/n \to 0$ in \eqref{AR}, it follows that
\begin{eqnarray*}
&& \frac{\Gamma(\theta+n+i\alpha)}{\Gamma(\theta+n)} \\
&&= n^{i \alpha} \left( 1 + \frac{\theta}{n} \right)^{i\alpha} \left\{ 1+ \frac{i \alpha(i \alpha -1)}{2n} \left(1+\frac{\theta}{n}\right)^{-1}  + O\left( \frac{1}{(\theta+n)^{2}}\right) \right\} \\
&&= n^{i \alpha} \left( 1 + \frac{i\alpha \theta}{n} + O\left(\frac{\theta^2}{n^2} \right) \right)
 \left\{ 1+ \frac{i \alpha(i \alpha -1)}{2n} \left(1 - \frac{\theta}{n} + O\left( \frac{\theta^2}{n^2} \right) \right)  + O\left( \frac{1}{n^{2}}\right) \right\} \\
&&= n^{i \alpha} \left\{ 1 + \frac{i \alpha \theta}{n} + \frac{i \alpha (i \alpha -1)}{2n} + O\left(\frac{\theta^2}{n^2}\right) \right\} \\
&&= n^{i \alpha} \left( 1 + \frac{i \alpha \theta}{n} + O\left(\frac{1}{n} + \frac{\theta^2}{n^2}\right) \right). 
\end{eqnarray*}
This completes the proof.
\end{proof}

\begin{lemma}\label{lem2}
As $\theta\to\infty$,
\[
\frac{\Gamma(\theta+1)}{\Gamma(\theta+1+i\alpha)}
= \theta^{-i \alpha} \left\{ 1 - \frac{i \alpha (i \alpha +1)}{2\theta} +  O\left(\frac{1}{\theta^{2}}\right) \right\} 
\]
for $i=1,2,\ldots$.
\end{lemma}

\begin{proof}
Using a similar argument as in the proof of Lemma~\ref{lem1}, we have
\[
\frac{\Gamma(\theta+1)}{\Gamma(\theta+1+i\alpha)}
= (\theta+1)^{-i\alpha} \left\{ 1 - \frac{i \alpha  (i \alpha -1)}{2(\theta+1)} + O\left(\frac{1}{\theta^{2}}\right) \right\}.
\]
Hence, it holds that
\begin{eqnarray*}
\frac{\Gamma(\theta+1)}{\Gamma(\theta+1+i\alpha)}
&=& \theta^{-i \alpha} \left(1+\frac{1}{\theta}\right)^{-i\alpha} \left\{ 1 - \frac{i \alpha  (i \alpha -1)}{2\theta}\left(1+\frac{1}{\theta}\right)^{-1} + O\left(\frac{1}{\theta^{2}}\right) \right\} \\
&=& \theta^{-i \alpha} \left(1 -\frac{i\alpha}{\theta} + O\left(\frac{1}{\theta^{2}}\right)  \right) \left\{ 1 - \frac{i \alpha  (i \alpha -1)}{2\theta} + O\left(\frac{1}{\theta^{2}}\right) \right\} \\ 
&=& \theta^{-i \alpha} \left\{ 1 - \frac{i \alpha (i \alpha +1)}{2\theta} +  O\left(\frac{1}{\theta^{2}}\right) \right\} .
\end{eqnarray*}
This completes the proof.
\end{proof}

\begin{lemma}\label{lem3}
As $\theta\to\infty$,
\[
\frac{\Gamma\left(\frac{\theta}{\alpha}+1+i\right)}{\Gamma\left(\frac{\theta}{\alpha}+1\right)}
= \left(\frac{\theta}{\alpha}\right)^{i} \left\{ 1 + \frac{i(i +1) \alpha}{2\theta} + O\left(\frac{1}{\theta^{2}} \right) \right\}
\]
for $i=1,2,\ldots$.
\end{lemma}

\begin{proof}
Using a similar argument as in the proof of Lemma~\ref{lem1}, we have
\[
\frac{\Gamma(\frac{\theta}{\alpha} +1+i)}{\Gamma( \frac{\theta}{\alpha} +1)}
= \left(\frac{\theta}{\alpha}+1 \right)^{i} \left\{ 1 + \frac{i (i -1)}{2(\frac{\theta}{\alpha} +1)} + O\left(\frac{1}{\theta^{2}}\right) \right\}.
\]
Hence, it holds that
\begin{eqnarray*}
\frac{\Gamma(\frac{\theta}{\alpha} +1+i)}{\Gamma( \frac{\theta}{\alpha} +1)}
&=& \left( \frac{\theta}{\alpha} \right)^{i} \left(1+\frac{\alpha}{\theta}\right)^{i} \left\{ 1 + \frac{i (i  -1) \alpha }{2\theta}\left(1+\frac{\alpha}{\theta}\right)^{-1} + O\left(\frac{1}{\theta^{2}}\right) \right\} \\
&=& \left( \frac{\theta}{\alpha} \right)^{i} \left(1 + \frac{i\alpha}{\theta} + O\left(\frac{1}{\theta^{2}}\right)  \right) \left\{ 1 + \frac{i (i  -1) \alpha }{2\theta} + O\left(\frac{1}{\theta^{2}}\right) \right\} \\ 
&=& \left( \frac{\theta}{\alpha} \right)^{i} \left\{ 1 + \frac{i (i  +1) \alpha}{2\theta} +  O\left(\frac{1}{\theta^{2}}\right) \right\} .
\end{eqnarray*}
This completes the proof.
\end{proof}

\begin{proof}[Proof of Lemma~\ref{keyp}]
From Lemmas~\ref{lem1}, \ref{lem2}, and \ref{lem3}, it follows that
\begin{eqnarray*}
\lefteqn{\frac{\Gamma\left(\frac{\theta}{\alpha}+1+i\right)}{\Gamma\left(\frac{\theta}{\alpha}+1\right)}
\frac{\Gamma(\theta+n+i\alpha)}{\Gamma(\theta+n)}
\frac{\Gamma(\theta+1)}{\Gamma(\theta+1+i\alpha)} }\\
&=& 
\left( \frac{\theta}{\alpha} \frac{n^\alpha}{\theta^\alpha} \right)^i
\left( 1 + \frac{i \alpha \theta}{n} + O\left(\frac{1}{n} + \frac{\theta^2}{n^2}\right) \right) 
\\&& \times
\left\{ 1 - \frac{i \alpha (i \alpha +1)}{2\theta} +  O\left(\frac{1}{\theta^{2}}\right) \right\} 
\left\{ 1 + \frac{i (i +1) \alpha}{2\theta} + O\left(\frac{1}{\theta^{2}} \right) \right\}
\\
&=&
\left( \frac{\theta}{\alpha} \frac{n^\alpha}{\theta^\alpha} \right)^i \left\{ 1+ \frac{i \alpha \theta}{n} + \frac{i^2 \alpha(1-\alpha)}{2\theta} + O \left( \frac{1}{\theta^2} + \frac{\theta^2}{n^2} \right)  \right\},
\end{eqnarray*}
where 
\[
O \left( \frac{1}{\theta^2} + \frac{1}{n} + \frac{\theta^2}{n^2} \right) 
= O \left( \frac{1}{\theta^2} + \frac{\theta^2}{n^2} \right) 
\]
is used.
This completes the proof.
\end{proof}

\section{Concluding remark}
Under the asymptotic regime
\begin{equation}\label{CR}
n\to\infty, \quad \theta\to\infty, \quad \frac{\theta^{2\alpha+1}}{n^{2\alpha}} \to 0, \quad
 \frac{\theta^2}{n} \to 0, 
\end{equation}
Theorem~\ref{mthA} yields
${\sf E}\left[ Z \right] \to 0$ and ${\sf E}\left[ Z^2 \right] \to \alpha(1-\alpha)$, where
\[ 
Z = \sqrt{\theta} \left[ \frac{\alpha K}{\theta \left\{ (\frac{n+\theta}{\theta})^\alpha - 1 \right\} } -1 \right].
\]
We thus expect that $Z$ (or asymptotically equivalent quantities to $Z$) has a non-degenerate limit distribution under \eqref{CR} or stronger regimes.
Deriving asymptotic properties of $Z$ under such regimes is a possible future direction.

\section*{Acknowledgments}
The author was supported in part by Japan Society for the Promotion of Science KAKENHI Grant Number 18K13454 and 21K13836.
This study was partly carried out when the author was a member of Graduate School of Arts and Sciences, the University of Tokyo.

\end{document}